\documentclass[a4paper,leqno,11pt]{amsart}

\usepackage{amsfonts,amssymb,verbatim,amsmath,amsthm,latexsym,textcomp,amscd}
\usepackage{latexsym,amsfonts,amssymb,epsfig,verbatim}
\usepackage{amsmath,amsthm,amssymb,latexsym,graphics,textcomp}
\usepackage{mathtools}
\usepackage{graphicx}
\usepackage{color}
\usepackage{url}
\usepackage{enumerate}
\usepackage[mathscr]{euscript}
\usepackage{tikz-cd}
\usetikzlibrary{shapes.geometric}
\usepackage{dsfont}
\usetikzlibrary{matrix}
\usepackage{hyperref}

\input xy

\xyoption{all}

\DeclarePairedDelimiter\ceil{\lceil}{\rceil}
\DeclarePairedDelimiter\floor{\lfloor}{\rfloor}

\theoremstyle{definition}
\newtheorem{theorem}{Theorem}[section]
\newtheorem{prop}[theorem]{Proposition}
\newtheorem{lemma}[theorem]{Lemma}
\newtheorem{cor}[theorem]{Corollary}
\newtheorem{defn}[theorem]{Definition}

\newtheorem{rmk}[theorem]{Remark}

\newtheorem{thm}[theorem]{Theorem}
\newtheorem{notation}[theorem]{Notation}

\newenvironment{myeq}[1][]
{\stepcounter{theorem}\begin{equation}\tag{\thetheorem}{#1}}
{\end{equation}}

\newcommand{\C}{{\mathbb C}}

\newcommand{\Hyp}{{\mathbb H}}

\newcommand{\Z}{{\mathbb{Z}}}
\newcommand{\R}{{\mathbb R}}

\newcommand{\rls}{{\mathbb R}}

\newcommand{\K}{\mathbb K}

\newcommand\DD{{\mathcal D}}

\newcommand\FF{{\mathcal F}}

\newcommand\LL{{\mathcal L}}
\newcommand\MM{{\mathcal M}}

\newcommand\PP{{\mathcal P}}

\newcommand\PMF{{\PP\kern-2pt\MM\FF}}

\newcommand\PML{{\PP\kern-2pt\MM\LL}}

\newcommand\ep{\epsilon}

\newcommand{\fsubd}{\mathrel{{\scriptstyle\searrow}\kern-1ex^d\kern0.5ex}}
\newcommand{\bsubd}{\mathrel{{\scriptstyle\swarrow}\kern-1.6ex^d\kern0.8ex}}
\newcommand{\fsubeq}{\mathrel{\raise-.7ex\hbox{$\overset{\searrow}{=}$}}}
\newcommand{\bsubeq}{\mathrel{\raise-.7ex\hbox{$\overset{\swarrow}{=}$}}}

\newcommand{\tsh}[1]{\left\{\kern-.9ex\left\{#1\right\}\kern-.9ex\right\}}

\newcommand{\Index}{\mbox{Index}}
\newcommand{\ind}{\mbox{ind}}
\newcommand{\res}{\mathit{res}}
\newcommand{\Ker}{\mbox{Ker}}
\newcommand{\coind}{\mbox{coind}}

\title{The index of certain Stiefel manifolds}

\author{Samik Basu}
\address{Stat-Math Unit, Indian Statistical Institute, Kolkata 700108, India.}
\email{samik.basu2@gmail.com, samikbasu@isical.ac.in}
\author{Bikramjit Kundu}
\address{Department of Mathematics,  IIT Roorkee, Roorkee, Haridwar Road, Uttarakhand, 247667, India}
\email{bikramju@gmail.com, bikramjit.pd@ma.iitr.ac.in}

\subjclass[2010]{Primary: 55M20, 55M35; Secondary: 52A35, 55N91.}
\keywords{Existence of equivariant maps, Stiefel manifolds, Fadell-Husseini index, equivariant cohomology.}

\begin{document}

\maketitle
\begin{abstract}
This paper computes the Fadell-Husseini index of Stiefel manifolds in the case where the group acts via permutations of the orthogonal vectors. The computations are carried out in the case of elementary Abelian $p$-groups. The results are shown to imply certain generalizations of the Kakutani-Yamabe-Yujobo theorem. 
\end{abstract}

\setcounter{section}{0}

\section{Introduction}
The existence of equivariant maps between spaces with $G$-action has been an important question in equivariant topology. The Borsuk-Ulam theorem states that there are no maps from $S^n$ to $S^{n-1}$ which are equivariant for the $C_2$-action on the spheres via the antipodal map. Ruling out equivariant maps involves invariants of $G$-spaces, and their computation in the relevant examples. 

One of the most useful invariants from this point of view is the index. Perhaps the first such index for $C_2$-spaces was defined by Yang \cite{Yan54}, in the context of Kakutani's theorem about proving that for functions on spheres, there are orthogonal vectors which are mapped to the same value. Following along these lines, a number of other indices are defined for $C_2$-spaces which are compared in \cite[Chapter 5]{Mat03}. For general $G$-spaces, the {\it Fadell-Husseini index} \cite{FaHu88} has been widely used, and is defined as the kernel of the map from the cohomology of $BG$ to the cohomology  of the Borel construction. 

In this paper, we consider Stiefel manifolds of $k$-tuples of orthonormal vectors :  $V_k\R^l$ (real), $W_{l,k}$ (complex) and $X_{l,k}$ (quaternionic). The groups $C_p^n$ act on the Stiefel manifold for $k=p^n$ via the homomorphism of $C_p^n \to \Sigma_{p^n}$, which is induced by the left action of $C_p^n$ on itself. We write 
$$H^\ast(BC_p^n;\Z/p)= \begin{cases} \Z/2[\mu_1,\cdots, \mu_n] & \mbox{if } p=2 \\ 
                                      \Z/p[u_1,\cdots, u_n, v_1,\cdots, v_n]/(u_1^2,\cdots, u_n^2) &\mbox{if } p \mbox{ is odd},\end{cases}$$
where $|\mu_j|=1$, $|u_j|=1$, and $|v_j|=2$. We obtain the following computations for the index in this case (see Theorems \ref{indV}, \ref{indV1}, \ref{indW}, \ref{indX}, \ref{indWXn}, \ref{comp-2}, \ref{indc2}).
\begin{itemize}
\item For odd primes $p$, $\Index_{C_p}(V_p\R^l)=\langle v^{(\ceil{\frac{l}{p-1}}-1)(p-1)}\rangle$, and when $n\geq 2$,  $\Index_{C_p^n}(V_{p^n} \R^l) \subset \langle v_1^{p-1},\cdots, v_n^{p-1}\rangle$ is generated by elements of degree divisible by $2(p-1)$. For the prime 2, we have $\Index_{C_2}(V_2\R^l)=\langle \mu^{l-1}\rangle$, and if $l$ is odd, $\ind_{C_2}(V_2\R^l) \geq l-1$.  
\item For odd primes $p$, $\Index_{C_p}(W_{l,p})=\langle v^{(\ceil{\frac{l}{p-1}}-1)(p-1)}\rangle$,  and when $n\geq 2$, $\Index_{C_p^n}(W_{l,p^n}) \subset \langle v_1^{p-1},\cdots, v_n^{p-1}\rangle$ is generated by elements of degree divisible by $2(p-1)$. For the prime 2, we have $\Index_{C_2}(W_{l,2})=\langle \mu^{2l-2}\rangle$, and if $l$ is odd, $\ind_{C_2}(W_{l,2}) \geq 2l-2$.  
\item For odd primes $p$, $\Index_{C_p}(X_{l,p})=\langle v^{(\ceil{\frac{2l}{p-1}}-1)(p-1)}\rangle$ if $p \mid \ceil{\frac{2l}{p-1}}-1$ and $=\langle v^{(\ceil{\frac{2l}{p-1}}-2)(p-1)}\rangle$ otherwise, and when $n\geq 2$, $\Index_{C_p^n}(X_{l,p^n}) \subset \langle v_1^{p-1},\cdots, v_n^{p-1}\rangle$ is generated by elements of degree divisible by $2(p-1)$. For the prime 2, we have $\Index_{C_2}(X_{l,2})=\langle \mu^{4l-4}\rangle$, and if $l$ is odd, $\ind_{C_2}(X_{l,2}) \geq 4l-4$.  
\end{itemize}

The computation of the index is usually done in the context of a concrete problem with geometric consequences.  In the context of the topological Tverberg problem, index computations played a crucial role (see  \cite{Oz87}, \cite{Vol00}, \cite{BPZ17},  \cite{Cra12}). For their uses in other geometric contexts see \cite{BPL15}, \cite{BBKV18}, \cite{Har05}, \cite{Pet97}, \cite{Mat03}. The index of Stiefel manifolds (with $C_2^k$ acting by multiplying the $k$-vectors by a sign) has been used to prove extensions of theorems of Rattray and Makeev \cite{BK12}. Our index computations yield the following conclusions for geometric problems which are inspired by Kakutani's Theorem and its generalizations (see Corollary \ref{Kakp2} and Theorems \ref{Kakpn}, \ref{KakCH}). 
\begin{itemize} 
\item Let $f:S^{l-1} \to \R^m$ for $l \geq (\floor{\frac{m}{2}}+1)(p-1)+1$ for an odd prime $p$. Then, there are orthogonal unit vectors $v_1,\cdots,  v_p$ in $\R^l$ such that $f(v_1)=f(v_2)=\cdots = f(v_p)$. 
\item Let $f:S^{l-1}\to \R^m$ for $l\geq (\frac{m}{2}+1)(p^k-1)  +1$ for an odd prime $p$, then there are $p^k$ orthogonal vectors $v_1,\cdots, v_{p^k}$ in $\R^l$ such that $f(v_1)=f(v_2) =\cdots = f(v_{p^k})$. 
\item Let $f:S(\C^l)=S^{2l-1} \to \R^m$ be a map, and suppose that $l \geq (\floor{\frac{m}{2}}+1)(p-1)+1$ for an odd prime $p$. Then, there are orthogonal unit vectors $v_1,\cdots,  v_p$ in $\C^l$ such that $f(v_1)=f(v_2)=\cdots = f(v_p)$. 
\item Let $f:S(\Hyp^l)=S^{4l-1} \to \R^m$ be a map, and suppose that 
$$l \geq \begin{cases} \frac{1}{2}(\floor{\frac{m}{2}}+1)(p-1)+1 & \mbox{if } p\mid \floor{\frac{m}{2}} +1 \\
              \frac{1}{2}(\floor{\frac{m}{2}}+2)(p-1)+1 & \mbox{if } p\nmid \floor{\frac{m}{2}} +1 \end{cases} $$ 
for an odd prime $p$. Then, there are orthogonal unit vectors $v_1,\cdots,  v_p$ in $\Hyp^l$ such that $f(v_1)=f(v_2)=\cdots = f(v_p)$. 
\end{itemize}
(See Theorem \ref{KakCH} for more results of the above type.)
\begin{notation}
Throughout the document, the coefficient ring for the cohomology computations are $\Z/p$ for some prime $p$. 
\end{notation}

\subsection{Organization} In section \ref{prelim}, we provide the general introduction to the definitions of the index and the notations used throughout the paper. In section \ref{hom}, we lay out the method employed in computation of the spectral sequence associated to the homotopy orbits of a Stiefel manifold. In section \ref{indcomp}, we carry out the main computations of the index, and finally, in section \ref{appl}, we point out the main applications of the computation. 

\section{Preliminaries} \label{prelim}
In this section, we introduce the basic notations, definitions and results to be used in the paper. We use the category of $G$-equivariant spaces, and the equivariant index to rule out maps between certain free $G$-spaces. Throughout $G$ will stand for a finite group. The main reference for the homotopy theoretic results of this section is \cite{May96}, and the equivariant index and its applications is \cite{Mat03}. 

The category of $G$-spaces has as objects topological spaces with $G$-action, and as morphisms the $G$-equivariant maps. An equivariant homotopy between $f,g: X\to Y$ is given by a map $X\times [0,1] \to Y$, where $G$ acts trivially on the interval $[0,1]$. A $G$-CW-complex is formed by attaching cells of the type $G/H\times \DD^n$ in increasing dimension. 
In the category of $G$-spaces, we write the $H$-fixed points functor as 
$$X\mapsto X^H$$
for $H$ a subgroup of $G$. The free $G$-spaces $X$ satisfy $X^H=\varnothing$ for all subgroups $H$ which are $\neq e$. The universal $G$-bundle $EG\to BG$ which classifies principal $G$-fibrations, satisfies 
$$EG^H = \begin{cases} \simeq  \ast & \mbox{if } H=e \\ 
                                            \varnothing & \mbox{if } H\neq e. \end{cases} $$
One further notes that for a finite  $G$-CW-complex with free action (this is equivalent to the fact that all the cells are of the form $G/e \times \DD^n$), there is a unique map to $EG$ up to equivariant homotopy.                                              

Let $V$ be an orthogonal $G$-representation. One defines $S^V$ to be the one-point compactification of $V$, and uses the orthogonal structure to define $G$-spaces 
$$D(V) :=\{ v \in V : ||v|| \leq 1\}$$ 
and  
$$S(V) :=\{ v \in V : ||v|| =1\}.$$
An easy observation is that $S(V)^H=S(V^H)$, which is the sphere in the vector space of $H$-fixed points of $V$. We use the following notations for certain important representations to be considered in the paper. These representation theory results may be found in \cite{Ser77}. 
\begin{notation}
We use the notation $\sigma$ to denote the sign representation of $C_2$. For odd primes $p$, the notation $\xi$ stands for the $2$-dimensional real representation of $C_p$ in which the chosen generator of $C_p$ acts by the rotation of angle $\frac{2\pi}{p}$. This is the underlying real representation of a complex $1$-dimensional representation which we also denote by $\xi$. The regular representation of a general group $G$ is denoted by $\rho_G$ which splits as a direct sum of a trivial $1$ dimensional representation and the reduced regular representation which is denoted by $\bar{\rho}_G$. The notation $\Sigma_n$ will stand for the symmetric group on $n$ elements. Permuting the standard basis vectors of $\R^n$ induces an $n$-dimensional representation of $\Sigma_n$ which splits as a direct sum of a trivial $1$-dimensional representation generated by the sum of all the vectors and the standard representation which is denoted by $W$.   
\end{notation}

Recall that as a complex representation the regular representation $\rho$ of a finite group $G$ is a direct sum of $\dim(\pi)$-many copies of $\pi$, as $\pi$ varies over the irreducible representations of $G$. For the group $C_2$, the only irreducible representations are the trivial representation $1$ and the sign $\sigma$, so that 
$$\rho_{C_2} = 1+\sigma,$$
and this formula is also true as real representations. We denote $\rho^\R_G$ as the real regular representation. For the group $C_p$ for an odd prime $p$, the set of complex irreducible representations are $\{1,\xi,\xi^2,\cdots, \xi^{p-1}\}$ where $\xi^r$ is the tensor product of $r$ copies of $\xi$, which is the vector space $\C$ on which the generator of $C_p$ acts via multiplication by $e^{\frac{2\pi i r}{p}}$. We thus have 
$$\rho_{C_p}=1+\xi+\xi^2 +\cdots +\xi^{p-1}$$ 
as complex representations. As real representations we have 
$$\rho^\R_{C_p}= 1+r(\xi)+r(\xi^2)+\cdots + r(\xi^{\frac{p-1}{2}}),$$
where $r$ stands for the underlying real representation of a complex representation. For the groups $C_p^n$, the set of irreducible representations are given by the tensor products of irreducible representations of $C_p$. We write in this case $\pi_j$ as the projection onto the $j^{th}$-factor, and $\xi_j=\pi_j^\ast \xi$, and then the regular representation splits as 
$$\rho_{C_p^n}=\sum_{1\leq i_j \leq p} \xi_1^{i_1}\otimes \cdots \otimes \xi_n^{i_n}.$$ 

Let $G$ be of order $n$, and consider the action of $G$ on itself by left multiplication. This induces a homomorphism $G\to \Sigma_n$. The pullback of the permutation representation of $\dim n$ of $\Sigma_n$ to $G$ is the regular representation $\rho$ and that of the standard representation is the reduced regular representation $\bar{\rho}$. A representation $V$ of $G$ may be used to construct a bundle on the classifying space $BG$ : $EG\times_G V \to BG$. We often denote this bundle also by the same notation as the representation, and so we often write characteristic classes of representations as elements of $H^\ast(BG)$. 

The main results in this paper rely on ruling out equivariant maps between certain $G$-spaces. For this purpose, we require suitable invariants of equivariant spaces which are computable in simple examples. The index is such an invariant which we recall now. For a $G$-space $X$, we denote the orbit space by $X/G$. However, homology and cohomology computations of the orbit space are not easy.  It is easier to make computations for the homotopy orbit space $X_{hG} : = EG\times_G X$. For free $G$-spaces, $X_{hG}\simeq X/G$. The main technique is the fibration 
$$X \to X_{hG} \stackrel{p_X}{\to} BG. $$

\begin{defn} \label{ind} \cite{FaHu88}
The {\it Fadell-Husseini index} $\Index_G(X)$ of a $G$-space $X$ is defined as $\Ker(p_X^\ast)$ where $p_X^\ast : H^\ast(BG) \to H^\ast(X_{hG})$.
\end{defn}
The {\it Fadell-Husseini index} is a very useful invariant from the point of view of computations. Although it is defined over any coefficient ring, we will restrict ourselves to $\Z/p$-coefficients for some prime $p$. Some basic properties of the index are (see \cite{BPZ17})
\begin{itemize}
    \item \textit{Monotonicity:} If $X\to Y$ is a $G$-equivariant map, then $\Index_G(Y)\subseteq \Index_G(X)$.
    \item\textit{Additivity:} If $(X_1\cup X_2,X_1,X_2)$ is an excisive triple of $G$-spaces, then 
    $$\Index_G(X_1) \Index_G(X_2) \subseteq \Index_G(X_1\cup X_2).$$
    \item \textit{Join:} Let $X$ and $Y$ be $G$-spaces, then $\Index_G(X) \Index_G(Y) \subseteq \Index_G(X*Y)$.
\end{itemize} 
The index has an interesting connection to $RO(G)$-graded equivariant cohomology when $G=C_p$ \cite{BaGh21}. 

In the case of the elementary Abelian groups $C_p^n$, the localization theorem \cite[Theorem IV.2.1]{May96} has the following consequence for the ideal valued index \cite[Cor.1, p. 45]{Hsi75}. 
\begin{prop} \label{1P}
Let $X$ be a finite dimensional $C_p^n$-CW-complex. The fixed point space $X^{C_p^n}\neq \varnothing$  $\iff$ $\Index_{C_p^n}(X)= 0$. 
\end{prop}
Let us deduce some consequences. A $C_p^n$-CW-complex $X$ which is finite dimensional, such that $X$ is fixed point free, that is $X^{C_p^n}=\varnothing$,  must have non-trivial ideal valued index. This means that for the fibration 
$$X\to X_{hC_p^n} \to BC_p^n,$$
the Serre spectral sequence has a non-trivial differential. From the multiplicative structure on the cohomology spectral sequence in the case the action of $C_p^n$ on $H^\ast(X;\Z/p)$ is trivial, the lowest degree class in $H^\ast (BC_p^n)$ must lie in the image of a transgression. This is also the lowest degree class in $\Index_G (X)$. 

Proposition \ref{1P} allows us to compute the index of representation spheres for the cyclic groups of prime order. In case of the group $C_p$, a fixed point free space must have free action. The cohomology of $BC_p$ may be written as 
$$H^\ast(BC_p;\Z/p) \cong \begin{cases} \Z/2[\mu] & \mbox{if } p=2 \\ 
\Z/p[u,v]/(u^2) & \mbox{if } p \mbox{ is odd}, \end{cases} $$
with $|\mu|=1$, $|u|=1$, and $|v|=2$. As $S(V)^{C_p}=S(V^{C_p})$ and with $\Z/p$-coefficients, $\Index_{C_p}(S(V))= 0$ $\iff$ $V$ contains a copy of the trivial representation. If $V^{C_p}=0$, then, $\dim(V)$ is even if $p$ is odd, as $V$ will be a sum of various $\xi^j$. As the fibre of $S(V)\to S(V)_{hC_p} \to BC_p$ is a sphere, there is only one possible non-trivial differential, namely $d_{\dim(V)}$. It follows that 
$$\Index_{C_p}(S(V))=\begin{cases} \langle \mu^{\dim(V)} \rangle &\mbox{if } p=2 \\ 
                                                       \langle v^{\frac{\dim(V)}{2}} \rangle & \mbox{if } p \mbox{ is odd}. 
\end{cases} $$ 

There are analogous results for elementary Abelian groups. In this paper, we use the index computations for the reduced regular representation. This involves the Dickson polynomials \cite{Dic11} for which we refer to \cite{MM82}. Recall that 
$$H^\ast(BC_p^n;\Z/p)= \begin{cases} \Z/2[\mu_1,\cdots, \mu_n] & \mbox{if } p=2 \\ 
                                      \Z/p[u_1,\cdots, u_n, v_1,\cdots, v_n]/(u_1^2,\cdots, u_n^2) &\mbox{if } p \mbox{ is odd}.\end{cases}$$

Denote the determinant 
\[
\begin{vmatrix}
v_{1}^{p^{i_1}} & v_{2}^{p^{i_1}} & \cdots & v_{n}^{p^{i_1}} \\ 
v_{1}^{p^{i_2}} & v_{2}^{p^{i_2}} & \cdots & v_{n}^{p^{i_2}} \\
\cdots \\
v_{1}^{p^{i_n}} & v_{2}^{p^{i_n}} & \cdots & v_{n}^{p^{i_n}}
\end{vmatrix}
\]
by $D(v_1,\cdots,v_n,p^{i_1},\cdots,p^{i_n})$. Note that this is a homogeneous element of $H^\ast (BC_p^n)$ lying in degree $2(p^{i_1}+\cdots + p^{i_n})$. 
\begin{defn}\cite{MM82}
$D_{n,j}=D(v_1,\cdots,v_n,1,p,p^2\cdots,\hat{p^j},\cdots,p^n)$
\end{defn}
Denote $D_{n,n}$ by $L_n$. Thus, the degree of $L_n$ is $2\cdot \frac{p^n - 1}{p-1}$. Note that $D_{n,0}=L_n^p$. 
\begin{lemma}\label{Dickson}\cite{MM82}
The polynomial $D_{n,j}$ is divisible by $L_n$, that is, $D_{n,j}=Q_{n,j}L_n$. Here, $Q_{n,j}$ is a non-zero polynomial invariant under $GL_n(\Z/p)$.
\end{lemma}

The polynomials $Q_{n,j}$ have degree $2(p^n-p^j)$, and are called Dickson polynomials. In fact, Dickson proved \cite{Dic11} 
\[ 
\Z/p[x_1,\cdots,x_n]^{GL_n(\Z/p)}=\Z/p[Q_{n,n-1}, Q_{n,n-2},\cdots, Q_{n,0}]. \]

Let  $c_i(C_p^n)$ be the Chern classes of bundle associated to the regular representation $\rho_{C_p^n}$ of $C_p^n$. From \cite{Mu75} and \cite{MM82} we have 
\[c_{i}(C_p^n)=
\begin{cases}
(-1)^{n+s}Q_{n,s} \quad &\text{for $i=p^n-p^s$} \\
0 \quad & \text{otherwise.}
\end{cases}
\]
If we consider the real regular representation $\rho^\R_{C_2^n}$ similar results hold for Stiefel-Whitney classes,
\[w_{i}(C_2^n)=
\begin{cases}
Q_{n,s}(\mu_1,\cdots,\mu_n) \quad &\text{for $i=2^n-2^s$} \\
0 \quad & \text{otherwise.}
\end{cases}
\]
Let \begin{myeq}\label{theta}
\theta=\begin{cases}
L_n(\mu_1,\cdots,\mu_n) \quad & \text{for $p=2$},\\
L_n^{\frac{p-1}{2}} \quad & \text{for $p>2$}.
\end{cases}
\end{myeq}
The definition gives $\deg \theta=p^n-1$. This is in fact the Euler class of reduced regular representation $\Bar{\rho}^\R$ of $C_p^n$ \cite{Vo92}. We also have the Fadell-Husseini index of the sphere inside the reduced regular representation in terms of $\theta$.
\begin{theorem}\cite{Vo92}
$\Index_{C_p^n}S((\Bar{\rho}^\R)^l)=\langle\theta^l\rangle$.
\end{theorem}

In this paper, we are interested in computations of the index for Stiefel manifolds. We fix the notation for this below. 
\begin{notation}
The Stiefel manifold, denoted $St_k(\K^l)$, where $\K=$ $\R$, $\C$, or $\Hyp$, is the space of ordered orthonormal $k$-tuples of vectors in $\K^l$. It is homeomorphic to the quotient space $U_\K(l)/U_\K(l-k)$, where $U_\K(l)$ denotes the unitary group of inner product preserving transformations. The group $\Sigma_k$ acts on $St_k(\K^l)$ freely by permuting the $k$ orthonormal vectors. 

We also write $St_k(\R^l)$ as $V_k(\R^l)$, $St_k(\C^l)$ as $W_{l,k}$, and $St_k(\Hyp^l)$ as $X_{l,k}$. 
\end{notation}

The space $EG$ may be filtered by skeleta $E^{(n)}G$, and this may be used to define a numerical index of $G$-spaces \cite[Chapter 6]{Mat03}. This depends on the choice of skeleta, and for the purposes of this paper, we recall the formulation for $G=C_2$, in which case the skeleta of $EC_2$ are given by $S((n+1)\sigma)$, the space $S^n$ with antipodal action. The Borsuk-Ulam theorem implies that there are no $C_2$-maps from $S((n+1)\sigma) \to S(n\sigma)$, which inspires the following definitions. 
\begin{defn}
Let $X$ be a finite $C_2$-CW-complex with free action. Then define 
$$\ind_{C_2}(X) = \min \{ n\geq 0  \mbox{ such that there exists a } C_2\mbox{-map } X \to S((n+1)\sigma) \}, $$
$$\coind_{C_2}(X)=\max \{ n\geq 0 \mbox{ such that there exists a } C_2\mbox{-map }  S((n+1)\sigma) \to X \}.$$
\end{defn}
We readily conclude that $\coind_{C_2}(S((n+1)\sigma))=\ind_{C_2}(S((n+1)\sigma))=n$ from the Borsuk-Ulam theorem, and that $\coind_{C_2}(X)\leq \ind_{C_2}(X)$. The existence $C_2$-map $X \to Y$ implies 
$$\ind_{C_2}(X) \leq \ind_{C_2}(Y), ~ \coind_{C_2}(X) \leq \coind_{C_2}(Y).$$
From the fact that $\Index_{C_2}(S((n+1)\sigma))= \langle \mu^{n+1} \rangle$, we see that  
$$\langle \mu^{\coind_{C_2}(X)+1} \rangle \subset \Index_{C_2}(X) \subset \langle \mu^{\ind_{C_2}(X)+1} \rangle.$$
The techniques in \cite[Proposition 5.3.2]{Mat03} have the following implications for $\ind_{C_2}$ and $\coind_{C_2}$.
\begin{prop}\label{conn}
1) Suppose that $X$ is a free $C_2$-CW-complex of dimension $\leq n$. Then, $\ind_{C_2}(X)\leq n$. \\
2) Suppose that $X$ is a free $C_2$-space which is $r$-connected. Then, $\coind_{C_2}(X) \geq r+1$. 
\end{prop}

\section{Homotopy orbits of Stiefel manifolds}\label{hom}

Let $St_k(\K^l)$ be the Stiefel manifold of a $k$-tuple of orthonormal vectors in $\K^l$ where $\K$ is one of $\R$, $\C$ or $\Hyp$. The orthogonal group $U_\K(k)$ acts on $St_k(\K^l)$ freely, where $U_\R(k)= O(k)$, $U_\C(k)=U(k)$, and $U_\Hyp (k)= Sp(k)$ the orthogonal, unitary, and symplectic groups respectively. For a subgroup $G$ of $U_\K(k)$, our goal is to compute the cohomology ring of $St_k(\K^l)_{hG}\simeq St_k(\K^l)/G$, and the map $H^\ast (BG) \to H^\ast (St_k(\K^l)_{hG})$.   

We note that $St_k(\K^l)=U_\K(l)/U_\K(l-k)$, and denote $G_k(\K^l)=U_\K(l)/U_\K(k)\times U_\K(l-k)$ as the Grassmannian of $k$-planes in $\K^l$. Our main technique involves the following diagram of fibrations \cite{Bor53}
\begin{myeq}\label{fibseqG}
\xymatrix{
St_k(\K^l) \ar@{=}[rr] \ar[d] & & St_k(\K^l) \ar[d] \\
St_k(\K^l)_{hG} \ar[rr] \ar[d] & & G_k(\K^l) \ar[d] \\
BG                    \ar[rr]           & & BU_\K(k).
}
\end{myeq}
We compute the Serre spectral sequence for the left column by using the commutative diagram between spectral sequences and the spectral sequence associated to the right column. The latter spectral sequence is computed via the following commutative diagram of fibrations
\begin{myeq}\label{fibseqGr}
\xymatrix{
U_\K(l) \ar[rr] \ar[d] & & St_k(\K^l) \ar[d] \\
G_k(\K^l) \ar@{=}[rr] \ar[d] & & G_k(\K^l) \ar[d] \\
B(U_\K(k)\times U_\K(l-k))  \ar[rr]       & & BU_\K(k).
}
\end{myeq}
The top row in the diagram above is given by the quotient map $U_\K(l)\to U_\K(l)/U_\K(l-k) = St_k(\K^l)$, and the bottom row by the projection onto the first factor. The diagram between the fibrations is implied by the fact that $G_k(\K^l)\cong St_k(\K^l)/U_\K(k) \cong U_\K(l)/U_\K(l-k)\times U_\K(k)$. In the complex case $St_k(\C^l)=W_{l,k}$, this has already been computed in \cite[Proposition 2.1]{BaSu2017} (using methods from \cite{Bor53}), which is rewritten below.  
\begin{prop}\label{C-stief}
The cohomology ring $H^\ast(W_{l,k})$ is the exterior algebra $\Lambda(y_{l-k+1},\cdots, y_l)$, with $|y_j|=2j-1$. For the spectral sequence associated to $W_{l,k}\to G_k(\C^l) \to BU(k)$, the classes $y_j$ are transgressive with $
d_{2j}( y_j)=-c_j'$, where $c_j'$ are defined by the equation $(1+c_1'+\cdots)(1+c_1+\cdots +c_k)=1$.  
\end{prop}

Analogous arguments work for the real and the quaternionic cases. In the quaternionic case, $H^\ast(BSp(n))\cong \Z[q_1,\cdots, q_n]$ where the $q_i$ in degree $4i$ are the symplectic Pontrjagin classes. The result for the quaternionic case is described in the Proposition below. 
\begin{prop}\label{H-stief}
The cohomology  $H^\ast(X_{l,k})$ is the exterior algebra $\Lambda(z_{l-k+1},\cdots, z_l)$, with $|z_j|=4j-1$. For the spectral sequence associated to $X_{l,k}\to G_k(\Hyp^l) \to BSp(k)$, the classes $z_j$ are transgressive with $
d_{4j}( z_j)=-q_j'$, where $q_j'$ are defined by the equation $(1+q_1'+\cdots )(1+q_1+\cdots +q_k)=1$.  
\end{prop}

In the real case we compute spectral sequences for the fibration 
$$V_k(\R^l) \to Gr_k(\R^l) \to BO(k)$$
with $\Z/2$-coefficients, and the fibration 
$$V_k(\R^l)\to \tilde{Gr}_k(\R^l) \to BSO(k)$$
with $\Z/p$-coefficients for odd primes $p$. In the former case, the result is similar to the Propositions above with the Stiefel-Whitney classes instead of the Chern classes. 
\begin{prop}\label{R-stief-2}
With $\Z/2$-coefficients, the cohomology ring $H^\ast(V_{k}(\R^l))$ is additively the exterior algebra $\Lambda(\omega_{l-k+1},\cdots, \omega_l)$, with $|\omega_j|=j-1$. For the spectral sequence associated to $V_{k}(\R^l) \to G_k(\R^l) \to BO(k)$, the classes $\omega_j$ are transgressive with $
d_{j}( \omega_j)=-w_j'$, where $w_j'$ are defined by the equation $(1+w_1'+\cdots)(1+w_1+\cdots +w_k)=1$.  
\end{prop}

In the latter case, $\tilde{Gr}_k(\R^l)$ is the oriented Grassmannian. We have \cite{Bor53}
\[
H^*(SO(l);\Z/p) \cong 
\begin{cases}
\Lambda_{\Z/p}(x_1,\cdots ,x_{\frac{l-1}{2}})  & \text{if $l$ is odd} \\ 
\Lambda_{\Z/p}(x_1,\cdots,x_{\frac{l-2}{2}},\ep_l)  & \text{if $l$ is even}
\end{cases}\]
where $|x_i|=4i-1$, $|\epsilon_l|=l-1$, and 
\[
H^\ast(BSO(l);\mathbb{Z}/p) \cong 
\begin{cases}
\Z/p[p_1,\cdots,p_{\frac{l-1}{2}}] & \text{if $l$ is  odd} \\
 \Z/p[p_1,\cdots,p_{\frac{l-2}{2}},e_l] & \text{if $l$ is even}.
\end{cases}\] 
One deduces 
\begin{myeq}\label{R-Stief-odd} 
H^*(V_k\R^l;\Z/p) \cong 
\begin{cases}
 \Lambda (x_{\frac{l-k+2}{2}},\cdots,x_{\frac{l-1}{2}},\sigma_{l-k}) & \mbox{if } k \mbox{ is odd}, l \mbox{ is odd} \\
\Lambda(x_{\frac{l-k+1}{2}},\cdots,x_{\frac{l-1}{2}}) & \mbox{if } k \mbox{ is even}, l \mbox{ is odd} \\
\Lambda(x_{\frac{l-k+1}{2}},\cdots,x_{\frac{l-2}{2}},\epsilon_l) & \mbox{if } k \mbox{ is odd}, l \mbox{ is even} \\
 \Lambda(x_{\frac{l-k+2}{2}},\cdots,x_{\frac{l-2}{2}},\sigma_{l-k},\epsilon_l) & \mbox{if } k \mbox{ is even}, l \mbox{ is even} 
\end{cases}
\end{myeq}
where $|\sigma_j|=j$. The spectral sequence computation in cases of interest is described in the Proposition below. 
\begin{prop}\label{p1}
Assume that $k$ is odd. For the spectral sequence associated to $V_k(\R^l) \to \tilde{Gr}_k(\R^l) \to BSO(k)$, the classes $x_i$  are transgressive and $d_{4i} (x_i)=-p_i'$ where $p_i'$ are defined by the equation $(1+p_1'+\cdots )(1+p_1+\cdots p_k)=1$. The classes $\sigma_{l-k}$ (for $l$ odd) and $\ep_l$ (for $l$ even) are permanent cycles.  
\end{prop}
\begin{proof} 
The first part of the proof follows exactly analogously as in \cite{BaSu2017}(Proposition 2.1). We write down the details in the case $l$ is even while the odd case is similar. In the spectral sequence for fibration 
\[ SO(l) \to \tilde{Gr}_k\R^l \to B(SO(k)\times SO(l-k)), \] 
the differentials are $d(x_i)=p_i(\gamma_k\oplus\gamma_{l-k})$ \cite{Bor53}. We now use the diagram of fibrations \eqref{fibseqGr}, noting that for $i>\frac{l-k}{2}$,  the classes $x_i$ pull back to the classes with the same notation under $q^\ast$ (where $q: SO(l) \to V_k(\R^l)$). We write $p_i$ for $p_i(\gamma_k)$ and $\tilde{p}_i$ for $p_i(\gamma_{l-k})$, so that, $p_i=0$ if $i>\frac{k}{2}$ and $\tilde{p}_j=0$ if $j>\frac{l-k}{2}$. The formula $d(x_i)=p_i(\gamma_k\oplus\gamma_{l-k})$ implies that in the $(2(l-k)+1)$-page, $\tilde{p}_i$ is equivalent to $p_i'$ for $i\leq\frac{l-k}{2}$. Thus
\begin{align*}
 d_{2(l-k+1)}(x_{\frac{l-k+1}{2}}) &= p_{\frac{l-k+1}{2}}(\gamma_k \oplus \gamma_{l-k})\\             
              &= \sum_{i+j=\frac{l-k+1}{2}} p_i\tilde{p}_j \\
              &= \sum_{i+j=\frac{l-k+1}{2},i\geq 1} p_ip_j' + \tilde{p}_{\frac{l-k+1}{2}} \\
              &= \sum_{i+j=\frac{l-k+1}{2},i\geq1} p_ip_j' \\
              &= -p_{\frac{l-k+1}{2}}'.\\        
\end{align*} 

It now remains to prove that the classes $\sigma_{l-k}$ and $\ep_l$ are permanent cycles in the relevant cases. 
When $l-k$ is even, the class $\sigma_{l-k}$ is the lowest degree non-trivial class which is transgressive for degree reasons. Further it's degree is even, which forces it to transgress to zero as the $\Z/p$-cohomology groups in $BSO(k)$ are zero in odd degrees.

For $l$ even, it follows by examining the fibration $SO(l)\to \tilde{Gr}_k(\R^l) \to BSO(k)\times BSO(l-k)$ that $\ep_l$ is transgressive, and $d_l(\ep_l)$ is the Euler class of $\gamma_k \oplus \gamma_{l-k}$. This is zero as $k$ is odd.  
\end{proof}

In order to compute with the Serre spectral sequence $St_k(\K^l) \to St_k(\K^l)_{hG} \to BG$ via the map of fibrations \eqref{fibseqG}, we note that Propositions \ref{C-stief}, \ref{H-stief}, \ref{R-stief-2}, \ref{p1} imply that the differentials are described by values of certain characteristic classes of the universal bundles. These universal bundles are the associated bundle over $BU_\K(k)$ induced by the action of $U_\K(k)$ on $\K^k$. Therefore, the differentials for $St_k(\K^l) \to St_k(\K^l)_{hG} \to BG$ are determined by the characteristic classes of the associated $G$-representation. 

\section{Index computations}\label{indcomp}

We compute the index of the Stiefel manifolds $St_k(\K^l)$ for the groups  $C_p^n$. The actions which we consider are induced by homomorphisms 
$$ C_p^n \to \Sigma_{p^n}\subset U_\K(p^n),$$
where $C_p$ is included as the cyclic subgroup generated by a $p$-cycle, and for $n\geq 2$, $C_{p^n}$ is included as the transitive action on itself by left multiplication. In either case, note that the induced representation on the groups $C_p^n\to GL_\K(p^n)$ is the regular representation over $\K$.

\subsection{Computations at odd primes}
  We start by assuming that $p$ is odd and that the group $G=C_p$. We also first consider the case $\K=\R$. The cohomology of $BC_p$ is given by 
  \[H^*(BC_p;\Z/p)\cong \Z/p[v]\otimes\Lambda_{\Z/p}[u]\]
   where $|u|=1, |v|=2$. 
  
 \begin{theorem}\label{indV}
Let  $m=\ceil{\frac{l}{p-1}} - 1$. Then,   $\Index_{C_p}V_p\R^l$ is the ideal $\langle v^{m(p-1)}\rangle $. 
 \end{theorem}
 \begin{proof} Consider the fibration $  V_{p}\R^l \to (V_{p}\R^l)_{hC_p} \to BC_p$. The associated spectral sequence has  $E_2$ page
 \[
 E_2^{s,t}(V_p\R^l)= H^s(BC_p;H^t(V_p\R^l;\Z/p)) \Rightarrow H^{s+t}((V_p\R^l)_{hC_p};\Z/p).
 \] 
 The $C_p$-action on $H^t(V_p\R^l;\Z/p)$ is easily observed to be trivial. For, $C_p$ acts on $V_p\R^l$ via the inclusion $C_p \subset SO(p)$, and the latter action on $V_p\R^l$ as the principal $SO(p)$ bundle $V_p\R^l \to \tilde{Gr}_p\R^l$. Now as $SO(p)$ is connected, the action of $C_p$ is through maps which are homotopic to the identity. Therefore we may write 
 \[
 E_2^{s,t}(V_p\R^l)=H^s(BC_p) \otimes H^t(V_p\R^l) \Rightarrow H^{s+t}((V_p\R^l)_{hC_p};\Z/p).
 \]
Recall that the cohomology of the Stiefel manifold is given by \eqref{R-Stief-odd}.

We note that the universal $p$-bundle over $BSO(p)$ pulls back to the bundle $V$ over $BC_p$ which is associated to the real regular representation of $C_p$.  Therefore, as bundles 
$$V=\oplus_{i=1}^{\frac{p-1}{2}} r(\xi^{i}) \oplus \ep_{\R}.$$ 
 We have the following formulae for Pontrjagin classes,
\[
 p(\xi)=1-v^2, \quad   p(\xi^i)=(1-(iv)^2)\]
\[
\Rightarrow p(V)= \prod_{i=1}^\frac{p-1}{2}(1-i^2v^2).
\] \\
Therefore $p_k(V)= (-1)^k\sigma_kv^{2k}$, where $\sigma_k$ is the $k$-th elementary symmetric polynomial in $\{ i^2|~ i=1,\cdots,\frac{p-1}{2}\}$. We note that in $\Z/p$ this set is precisely the set of quadratic residues, and hence are the $\frac{p-1}{2}$ zeroes of the polynomial $x^{\frac{p-1}{2}}-1$. Hence, $\sigma_k=0$ for $k=1,\cdots,\frac{p-1}{2}$ and $\sigma_\frac{p-1}{2}=1$, and this implies $p(V)=1-(-1)^{\frac{p-1}{2}}v^{p-1}$. 

The diagram \eqref{fibseqG} yields  the following commutative diagram \\
\[
\xymatrix{
     V_{p}\R^l \ar@{=}[rr] \ar[d] & & V_{p}\R^l \ar[d] \\
     (V_{p}\R^l)_{hC_p} \ar[rr] \ar[d] & & \tilde{Gr}_{p}(\R^l) \ar[d] \\
     BC_p \ar[rr] & & BSO(p). }
\] 
 The commutative diagaram allows us to read off the differentials in the spectral sequence for
\begin{myeq}\label{D5}
V_{p}\R^l \to (V_{p}\R^l)_{hC_p} \to BC_p 
\end{myeq}
from the Serre spectral sequence for 
\[
V_{p}\R^l \to \tilde{Gr}_p(\R^l) \to BSO(p). 
\]
The differentials for the associated Serre spectral sequence for the fibration (\ref{D5}) are induced by according to Proposition (\ref{p1}) 
\[ d_{4i}(x_i)=-p_i', \quad d_j(\epsilon_l)=0,\quad \mbox{and }\quad d_j(\sigma_{l-p})=0. \] 
It follows that in \eqref{D5}, $x_i$, $\ep_l$ and $\sigma_{l-p}$ are transgressive, and they  satisfy 
\[ d_{4i}(x_i)=-p_i'(V), \quad d_j(\epsilon_l)=0,\quad \mbox{and }\quad d_j(x_{l-p})=0, \] 
where the classes $p'(V)$ are defned by the equation $p(V)p'(V)=1$. We have computed that 
\[p(V)=1-\pm v^{p-1}, \] 
which implies
\[
p'(V) = 1+\sum_{k\geq 1}\pm v^{k(p-1)}. 
\]
Note that whenever $l\geq p$ and $m(p-1)< l \leq (m+1)(p-1)$ there is exactly one $x_{\frac{i}{2}}$ in $H^*(V_p\R^l;\Z/p)$ such that $i$ is divisible by $p-1$, and we have $\frac{i}{p-1}=m$. Therefore, the first non-trivial differential in \eqref{D5} is 
\[ d_{2m(p-1)}(x_{m(\frac{p-1}{2})})=\pm v^{m(p-1)}. \] 
Hence, $\Index_{C_p}V_p{\R^l}$ is the ideal generated by $v^{m(p-1)}$.
\end{proof}

Next we consider $\K=\C$, and observe that the computation is the same as the real case. 
\begin{theorem}\label{indW}
Let  $m=\ceil{\frac{l}{p-1}}-1$. Then,   $\Index_{C_p}W_{l,p}$ is the ideal $\langle v^{m(p-1)}\rangle $. 
\end{theorem}
\begin{proof}
Proceeding analogously as in Theorem \ref{indV}, we observe that the representation induced by the inclusion $C_p \to U(p)$ is the regular representation. Therefore, the canonical $p$-plane bundle pulls back to the complex vector bundle $W$ over $BC_p$ induced by the regular representation. The differentials in the spectral sequence for 
\[ 
W_{p}\C^l \to (W_{p}\C^l)_{hC_p} \to BC_p
\]
are determined by $c'(W)$. Note that  $W=\oplus_{i=1}^{p-1}\xi^i \oplus \epsilon_{\C}$. We thus have,
\[c_1(\xi^i)=iv, ~ c(W)=\prod_{i=1}^{p-1}(1+iv).\]
\[\implies c(W)=\prod_{i=1}^{\frac{p-1}{2}}(1-i^2v^2)=1-(-1)^{\frac{p-1}{2}}v^{p-1}.\]
\[\implies c'(W) = 1+\sum_{k\geq 1} \pm v^{k(p-1)}. \]
The arguments of Theorem \ref{indV} now imply $\Index_{C_p}(W_{l,p})=\langle v^{m(p-1)}\rangle$. 
\end{proof}
 
Finally when $\K=\Hyp$, the index is described below.  
\begin{theorem}\label{indX}
Let  $m$ be defined by 
$$m= \begin{cases} \ceil{\frac{2l}{p-1}} -2 & \mbox{if } p\nmid \ceil{\frac{2l}{p-1}} -1 \\ 
                              \ceil{\frac{2l}{p-1}} -1 & \mbox{otherwise}. \end{cases} $$
 Then,   $\Index_{C_p}X_{l,p}$ is the ideal $\langle v^{m(p-1)}\rangle $. 
\end{theorem}
\begin{proof}
The quaternionic bundle $U$ over $BC_p$ induced by the map $C_p\to Sp(p)$ comes from the quaternionic regular representation. As the symplectic Pontrjagin classes are up to sign the even Chern classes of the underlying complex bundle, which for $U$ is two copies of the regular representation, we have 
$$q(U)=  (1- (-1)^{\frac{p-1}{2}} v^{p-1})^2 $$
$$\implies q'(U) = (1+ \sum_{k\geq 1} (-1)^{\frac{k(p-1)}{2}} v^{k(p-1)})^2=   1+ \sum_{k\geq 1} (-1)^{\frac{k(p-1)}{2}}(k+1) v^{k(p-1)}. $$
As in Theorems \ref{indV} and \ref{indW}, the index is generated by $v^{k(p-1)}$ where $k$ is the smallest integer between $2(l-p+1)$ and $2l$ which is a multiple of $(p-1)$ and $p\nmid k+1$. This is precisely the $m$ described in the statement. 
\end{proof}
 
\bigskip

We next switch out attention to  the elementary Abelian groups $C_p^n$ for $n\geq 2$. Our main feature in these computations is to prove certain bounds on the index which turn out to be useful in the examples in the following section. Recall that the group cohomology may be described by 
$$H^\ast (BC_p^n; \Z/p) \cong \Z/p[v_1,\cdots, v_n]\otimes \Lambda_{\Z/p}(u_1,\cdots, u_n)$$
with $|u_i|=1$ and $|v_j|=2$.  We prove that the index is contained in the ideal generated by $(p-1)$-powers of the $v_j$.  
We first prove the following useful proposition.
\begin{prop}\label{p3}
Suppose $A_k$ is the set of all linear polynomials in $\Z/p[x_1,\cdots,x_k]$. Then,
\[
1-\prod_{l\in A_k}(1+l) = (-1)^{k-1}\sum_{i=0}^{k-1}(-1)^iQ_{k,i}.
\] 
Further, the Dickson polynomials $Q_{k,i}$ belong to the ideal $I_{p-1}(k):=\langle x_1^{p-1}, \cdots, x_k^{p-1}\rangle$.
\end{prop}
\begin{proof}
First let us examine the expression 
\[1-\prod_{l\in A_k}(1+l)\] 
which is related to the Chern classes of complex regular representation of $C_p^k$\cite[Eq. 3.2, Th. 6.2]{MM82}. So we can rewrite the expression as \begin{align*}1-\prod_{l\in A_k}(1+l)=&1-c(\rho_{C_p^k})\\
=&(-1)^{k-1}\sum_{i=0}^{k-1}(-1)^i Q_{k,i}.
\end{align*}

For the remaining statement, it suffices to prove that $1-\prod_{l\in A_k}(1+l)$ belongs to the ideal $I_{p-1}(k)$.  We prove this by induction on $k$. For $k=1$, we have
\[
\prod_{a \in \Z/p} (1+ax)= \prod_{a \in (\Z/p)^\times}a\prod_{a \in (\Z/p)^\times}(a^{-1}+x)=-(x^{p-1}-1),
\]
which verifies the proposition in this case. Assuming the result for $k-1$, we write 
\[
\prod_{l\in A_k}(1+l)=\prod_{l\in A_{k-1}}(1+l)\prod_{a \in (\Z/p)^\times}(1+l+ax_k).
\]
Now we use
\[
\prod_{a \in \Z/p}(z+ax) = \prod_{a \in \Z/p}(z-ax)=x^p\prod_{a \in \Z/p}(z/x-a)=x^p((z/x)^p-z/x)=z^p-zx^{p-1}=z(z^{p-1}-x^{p-1}).
\]
This allows us to simplify the above as 
\[
\prod_{l\in A_{k-1}}\prod_{a \in \Z/p}(1+l+ax^k)= \prod_{l\in A_{k-1}}(1+l)((1+l)^{p-1}-x_k^{p-1}).
\]
Thus, modulo $I_{p-1}(k)$ the expression is equivalent to  $(\prod_{l \in A_{k-1}}(1+l))^p$. By induction hypothesis, the latter product is congruent to $1$ modulo $I_{p-1}(k-1)$. 
%
\end{proof}

We now return to the index computation of Stiefel manifolds starting with the real case below. For this, define the polynomials $S_i \in \Z/p[v_1,\cdots, v_n] \subset H^\ast(BC_p^n)$ by the equation 
\begin{myeq}\label{Seqn} 
(1+S_1+S_2+\cdots) (1+\sum_{i=0}^{n-1} (-1)^{n+i} Q_{n,i})^2 = 1 
\end{myeq}
with $|S_i|=4i$.   
\begin{theorem}\label{indV1}
The Fadell-Husseini index of $V_{p^n}\R^l$ is given by the formula 
\[\Index_{C_p^n}V_{p^n}\R^l= \langle S_i \mid \frac{l-p^n+1}{2}\leq i \leq \frac{l-1}{2}\rangle.\]
 The first non-trivial element of the index must lie in degree a multiple of $2(p-1)$.
\end{theorem}
\begin{proof}
 The second statement follows from the first and the formula for the degree of $Q_{k,i}$. Proceeding as in Theorem \ref{indV}, we compute the spectral sequence associated to the fibration  
\begin{myeq} \label{specVpn}
 V_{p^n}\R^l \to  (V_{p^n}\R^l)_{hC_p^n} \to BC_p^n, 
 \end{myeq}
which is related to the Pontrjagin classes of the associated bundle $V$ to the regular representation over $C_p^n$. We thus have 
$$V=r(\sum_{1\leq i_j \leq p} \xi_1^{i_1} \otimes \cdots \otimes \xi_n^{i_n})$$
 where $\xi_j=\pi_j^\ast (\xi)$, the pullback of $\xi$ along the $j^{th}$ projection $\pi_j:C_p^n \to C_p$, and $r$ stands for the underlying real representation. We thus have the following formula on the total Chern classes and Pontrjagin classes,
\[
c_1(\xi_1^{i_1} \otimes \cdots \otimes \xi_n^{i_n})=i_1v_1+...+i_nv_n,\  p(r(\xi_1^{i_1} \otimes \cdots \otimes \xi_n^{i_n}))=(1-(i_1v_1+....+i_nv_n)^2) \]
\[
\Rightarrow p(V)= \prod_{i_1,..i_n}(1-(\sum i_jv_j)^2) = c(\rho_{C_p^n})^2 = (1+\sum_{i=0}^{k-1} (-1)^{k+i} Q_{k,i})^2.
\]
Now Proposition \ref{p1} and the commutative diagram of fibrations \eqref{fibseqG}\\
\[
\xymatrix{
     V_{p^n}\R^l \ar@{=}[rr] \ar[d] & & V_{p^n}\R^l \ar[d] \\
     (V_{p^n}\R^l)_{hC_p^n} \ar[rr] \ar[d] & & \tilde{Gr}_{p^n}(\R^l) \ar[d] \\
     BC_p^n \ar[rr] & & BSO(p^n) } 
\]
implies the result, as the map $C_p^n \to SO(p^n)$ induces the regular representation on $C_p^n$. 
\end{proof}

Although we do not have an easy computation of the index, the result forces that the first non-zero class in the index lies in a degree divisible by $p-1$. This will be used in the applications in the following section. We have the following analogous result in the complex and the quaternionic case. We define  $M_i \in \Z/p[v_1,\cdots, v_n] \subset H^\ast(BC_p^n)$  by the equation 
\[ (1+M_1+M_2+\cdots) (1+\sum_{i=0}^{n-1} (-1)^{n+i} Q_{n,i})= 1 \]
with $|M_i|=2i$.  In terms of this notation we have the following result.
\begin{theorem}\label{indWXn}
The Fadell-Husseini index of the complex Stiefel manifold is given by 
$$\Index_{C_p^n}(W_{l,p^n})= \langle M_i \mid l-p^n + 1 \leq i \leq l-1\rangle.$$ 
For the quaternionic case the result is 
$$\Index_{C_p^n}(X_{l,p^n})= \langle S_i \mid l-p^n +1 \leq i \leq l \rangle.$$ 
In both cases, the first non-trivial element in the Index lies in degree a multiple of $2(p-1)$.
\end{theorem}

\begin{proof}
As in Theorem \ref{indV1}, we apply Propositions \ref{C-stief} and \ref{H-stief}, and this involves the computation of Chern classes  in the complex case, and the symplectic Pontrjagin classes in the quaternionic case, for the pullback of the universal bundle to $BC_p^n$. Here $W$, the pullback of the canonical $n$-dimensional complex bundle to $BC_p^n$, is induced by the regular representation, and is thus 
$$W= \sum_{1\leq i_j \leq p} \xi_1^{i_1}\otimes \cdots \otimes \xi_n^{i_n},$$
and
$$c(W)=c(\rho_{C_p^n})= (1+\sum_{i=0}^{n-1} (-1)^{n+i} Q_{n,i}).$$ 
It follows that the classes $c'(W)$ equal the classes $M_i$. In the quaternionic case, the pullback $U$ of the canonical bundle to $C_p^n$, is twice the regular representation as a complex bundle, 
and the rest of the proof follows analogously by Propositions \ref{p3} and \ref{H-stief}. 
\end{proof}

\begin{rmk}
In the literature \cite{Vol92}, we find the assertion that for $1\leq b\leq p^n-1$, 
\[\theta^{2a} \in \Index_{C^n_p}V_{p^n}\R^{a(p^n-1)+b}.\]
If this were true, one would obtain stronger bounds for the results of the following section. However if $n>1$, this is true only for small values of $a$. We let  $n=2$, $a=p$ and $b=p^2-1$. The image of first non-trivial differential lies in degree $2p(p^2-1)$ in the spectral sequence for the fibration \eqref{specVpn}, which is generated by a transgression onto a linear combination of monomials in $Q_{2,0}$ and $Q_{2,1}$. Theorem \ref{indV1} implies that this equals $S_{\frac{p(p^2-1)}{2}}$. By \eqref{Seqn}, we get 
\[ 
\begin{aligned}
1+S_1+S_2+\cdots &= (1+Q_{2,0} - Q_{2,1})^{-2}  \\ 
 &= \sum_{i\geq 0} (i+1)(Q_{2,1}-Q_{2,0})^i .
\end{aligned}\]
For degree reasons, the only possible monomials of $Q_{2,0}$ and $Q_{2,1}$ in degree $2p(p^2-1)$ are $Q_{2,0}^p$ and $Q_{2,1}^{p+1}$. It follows that 
\[ S_{\frac{p(p^2-1)}{2}} = -Q_{2,0}^p + 2 Q_{2,1}^{p+1} = 2 Q_{2,1}^{p+1} - \theta^{2p}.\]
Thus $\theta^{2p} \notin \Index_{C_p^2}(V_{p^2}\R^{p(p^2-1)+(p^2-1)})$.
%

However, if $a<p$, then for degree reasons we obtain that $S_{\frac{a(p^n-1)}{2}}$ is a multiple of $Q_{n,0}^a=\theta^{2a}$. Therefore, if $a<p$, $\theta^{2a} \in \Index_{C_p^n}(V_{p^n}\R^{a(p^n-1)+b})$ for $1\leq b \leq p^n-1$.
\end{rmk}


\bigskip

\subsection{Computations at $p=2$}
We follow the same method at the prime $2$. The cohomology ring $H^*(BC_2;\Z/2)=H^*(\R P^\infty;\Z/2)=\Z/2[\mu]$ where $\mu$ is the first Stiefel-Whitney class of canonical line bundle over $\R P^\infty$. 

\begin{theorem}\label{comp-2}
 As an ideal in $H^\ast(BC_2;\Z/2)$ we have the following computations, \\
1) $\Index_{C_2}V_2\R^l = \langle \mu^{l-1}\rangle$.\\ 
2)    $\Index_{C_2}W_{l,2} = \langle \mu^{2(l-1)}\rangle$.\\
3) $\Index_{C_2}X_{l,2} = \langle \mu^{4(l-1)}\rangle$.
\end{theorem}
\begin{proof}
We denote by $V$ (respectively $W$, $U$) the real (respectively complex, quaternionic) bundle over $BC_2$ associated to the regular representation. As the regular representation is written as the trivial plus the sign representation, we easily observe the following formulae. 
$$w(V)= 1+\mu.$$
$$c(W)=1+\mu^2.$$
$$q(U)=1+\mu^4.$$

In the real case, we compute the spectral sequence for the fibration 
\[  
 V_{2}\R^l \to (V_{2}\R^l)_{hC_2} \to BC_2 
 \] 
 using the map of fibrations \eqref{fibseqG}
\[
\xymatrix{
     V_{2}\R^l \ar@{=}[rr] \ar[d] & & V_{2}\R^l \ar[d] \\
     (V_{2}\R^l)_{hC_2} \ar[rr] \ar[d] & & {Gr}_{2}(\R^l) \ar[d] \\
     BC_2 \ar[rr] & & BO(2). }
\]
Applying Proposition \ref{R-stief-2}, the differentials are determined by the formal inverse of $w(V)=1+\mu$ which is $1+\sum_{k\geq 1} \mu^k$. As $H^\ast(V_2(\R^l);\Z/2) \cong \Lambda_{\Z/2}(x_{l-2},x_{l-1})$, the first non-zero differential is $d_{l-1}(x_{l-2})=\mu^{l-1}$. This completes the proof in the real case. 

The complex and the quaternionic cases follow analogously using Propositions \ref{C-stief} and \ref{H-stief}. 
\end{proof}

Note that if $X$ is a free $C_2$-space which is $(r-1)$-connected, then there is a $C_2$-map from $S(r\sigma)\to X$ by $C_2$-equivariant obstruction theory (Proposition \ref{conn}). This implies that $\Index_{C_2}(X) \subset \Index_{C_2}(S(r\sigma))=\langle \mu^{r+1} \rangle$.  Now note that $V_2(\R^l)$ is $(l-3)$-connected, so it already follows that the index does not have any classes in degree $<l-1$. Theorem \ref{comp-2} implies that the condition imposed by the connectivity hypothesis is sharp. This is also true for the complex and quaternionic Stiefel manifolds as $W_{l,2}$ is $(2l-4)$-connected and $X_{l,2}$ is $(4l-6)$-connected. 

One may now proceed with the computation for the group $C_2^n$ along the line above. We have $H^\ast (BC_2^n;\Z/2)\cong \Z/2[\mu_1,\cdots, \mu_n]$ with $|\mu_j|=1$ is the image of $\mu$ under the $j^{th}$-projection map. The total Stiefel Whitney class (and also the $\pmod{2}$ reduction of the Chern classes and the symplectic Pontrjagin classes) equals the product of all $1+l$ ( or a power of this product in case of Chern or symplectic Pontrjagin classes) where $l$ varies over the linear combinations of the $\mu_i$. Then, the restriction imposed by Proposition \ref{p3} is trivial in this case, and we do not obtain any new restrictions on the index other than the ones imposed by the connectivity hypothesis.  

While the Fadell-Husseini index is not very useful in ruling out equivariant maps in the case of $C_2$ or $C_2^n$, the bound on the $C_2$-index $\ind_{C_2}$ may be improved by $1$ in some cases. We use Steenrod operations to see this. 
\begin{theorem}\label{indc2} Let $l$ be odd. Then, $\ind_{C_2}(V_2(\R^l))\geq l-1$,  $\ind_{C_2}(W_{l,2})\geq 2l -2$, and
 $\ind_{C_2}(X_{l,2})\geq 4l-4$.  
\end{theorem}

\begin{proof}
Applying Theorem \ref{comp-2}, it remains to rule out $C_2$-maps $V_2(\R^l) \to S((l-1)\sigma)$, $W_{l,2} \to S((2l-2)\sigma)$ and $X_{l,2}\to S((4l-4)\sigma)$. We recall from \cite[Ch. IV, \S 10]{Whi78} that  
\begin{itemize}
\item The $l$-skeleton of $V_2(\R^l)$ is $\simeq$ to $\R P^{l-1}/\R P^{l-3}$. 
\item The $2l$-skeleton of $W_{l,2}$ is $\simeq$ to $\Sigma (\C P^{l-1}/\C P^{l-3})$. 
\item The $4l$-skeleton of $X_{l,2}$ is $\simeq$ to $\Sigma^3 (\Hyp P^{l-1}/ \Hyp P^{l-3})$. 
\end{itemize}
Suppose that there is a $C_2$-map $f: V_2(\R^l) \to S((l-1)\sigma)$. This induces a commutative diagram of fibrations 
$$\xymatrix{ V_2(\R^l) \ar[rr]^f \ar[d] & & S^{l-2} \ar[d] \\ 
          V_2(\R^l)_{hC_2} \ar[rr]^-{f_{hC_2}} \ar[d] && S((l-1)\sigma)_{hC_2} \ar[d]\\ 
           BC_2 \ar@{=}[rr]    &  & BC_2. }$$ 
This induces  a map of spectral sequences. The spectral sequence of the right column is determined by the differential $d_{l-1}(\ep_{l-2})=\mu^{l-1}$. From Theorem \ref{comp-2} we also have that $d_{l-1}(\omega_{l-2})=\mu^{l-1}$ in the spectral sequence of the left column. In both the spectral sequences, this is the first non-trivial differential. It follows that $f^\ast(\ep_{l-2}) = \omega_{l-2}$. 

On the other hand as $l$ is odd, so is $l-2$, which implies that the degree $l-2$ class in $H^\ast(\R P^{l-1};\Z/2)$ is mapped by $Sq^1$ to the degree $l-1$ class. It follows that in $H^\ast(V_2(\R^l);\Z/2)$, $Sq^1(\omega_{l-2})=\omega_{l-1}$. As a consequence we have 
$$f^\ast Sq^1(\ep_{l-2})=0$$
while
$$Sq^1f^\ast(\ep_{l-2})=Sq^1(\omega_{l-2})=\omega_{l-1}\neq 0,$$ 
a contradiction.  This completes the proof in the real case. 

The complex and the quaternionic cases are similar to the above by computing the Steenrod operations using the fact that they commute with the suspension. Under the given hypothesis, $Sq^2$ maps the generator of degree $2l-3$ to the generator in degree $2l-1$ in $H^\ast(\Sigma(\C P^{l-1}/\C P^{l-3}; \Z/2)$, and $Sq^4$ maps the generator of degree $4l-5$ to the generator of degree $4l-1$ in $H^\ast(\Sigma^3(\Hyp P^{l-1}/\Hyp P^{l-3});\Z/2)$.  The result follows analogously.
\end{proof}

\section{Applications}\label{appl}
In this section, we describe some applications for our index computations. Our application is oriented towards the Kakutani's theorem in geometry and it's generalizations. The Kakutani theorem in geometry states  : \textit{For any convex body in $\R^l$ there is an $l$-cube such that the convex body is inscribed in the $l$-cube.} This was proved by Kakutani \cite{Kak42} for $n=3$, Yamabe and Yujobo \cite{YaYu50} for $n\geq 4$. It is implied by the following theorem
\begin{theorem}\label{kak}
 For any map $f:  S^{l-1} \to \R$, there is an orthonormal basis of vectors $(v_1,\cdots,v_l)$ of $\R^l$ such that $f(v_i)=f(v_j)$ for every $i$ and $j$.
\end{theorem}
 The above statement was proved by Yang \cite{Yan54}.

We make slightly different approach to this, which presents us with a problem in equivariant homotopy theory. From above mentioned $f$, we construct a map $F: V_l\R^l \to \R^l$ as
\[F(v_1,\cdots,v_l) = (f(v_1),\cdots,f(v_l)).\]
The map $F$ is $\Sigma_l$-equivariant. That is, we have a $\Sigma_l$-action on $V_l\R^l$ by 
\[(\sigma,(v_1,\cdots,v_l)) \mapsto (v_{\sigma(1)},\cdots,v_{\sigma(l)}) \] 
and on $\R^l$ by permuting coordinates. Let $W$ be the standard representation of $\Sigma_l$, so that as a representation the latter is $W\oplus \epsilon$, where $\epsilon$ is the trivial one dimensional representation. We compose $F$ with the projection onto $W$ so that we have a $\Sigma_l$-equivariant map $\tilde{F} : V_l\R^l \to W$ and we note,

\begin{prop} \label{stdred}
 $\tilde{F}(v_1,\cdots,v_l)=0$  iff $f(v_1)=f(v_2)=\cdots =f(v_l)$.
\end{prop}
\begin{proof}
Note that $\tilde{F}(v_1,\cdots,v_l)=0$ iff $F(v_1,\cdots,v_l) \in \epsilon$. The trivial representation $\epsilon$ is embedded as the diagonal in $\R^l$, so the result follows. 
\end{proof}

If $l$ is an odd prime $p$, a counter-example to Kakutani-Yamabe-Yujobo theorem gives a $\Sigma_p$ equivariant map $V_p\R^p \to W$ which is non-zero at every point. We restrict to a $p$-cycle $C_p$, and also apply the deformation retraction of $W-0$ to the unit sphere $S(W)$. Note that the standard representation $W$ restricts to the reduced regular representation $\bar{\rho}$ of $C_p$. This gives a $C_p$-map $V_p(\R^p) \to S(\bar{\rho})$. We compare the Fadell-Husseini indices 
$$\Index_{C_p}(S(\bar{\rho}))=\langle v^{\frac{p-1}{2}} \rangle, \quad \mbox{and } \Index_{C_p}(V_p\R^p) = \langle v^{p-1} \rangle $$
by Theorem \ref{indV}. This rules out any possible $C_p$-map from $V_p\R^p \to S(W)$, implying the result in this case.  For $l=p^n$ an odd prime power, we may use the map $C_p^n \to \Sigma_{p^n}$ induced by the action of $C_p^n$ on itself by left multiplication. In this case we may not deduce the result from the index computations on account of our significantly weaker Theorem \ref{indV1}. Finally, if $l$ is not a prime power, restricting to various subgroups of $\Sigma_l$ is not useful on account of the existence of maps from $EG$ to $S(\res_G(W))$ proved in \cite{BaGh2017} for many examples of $G$.  

In the following we seek a generalization of Kakutani's theorem replacing the target $\R$ by $\R^m$. The results will follow the pattern above : the strongest results for odd primes $p$, significantly weaker results for prime powers, and no results for non prime powers. The problem we consider here is -- \\
\textbf{Question:} \textit{Find the integer $l(m,n)$ such that for $l \geq l(m,n)$ and any map $f : S^l \to \rls^m$ there are  $n$-mutually orthogonal vectors to the same value.} 

Given a $f$ as in the question, we proceed analogously as above defining $F:V_n \R^l \to \R^{mn}$ by 
$$F(v_1,\cdots,v_n)=(f(v_1),\cdots, f(v_n)).$$
We put the $\Sigma_n$ action on $V_n\R^l$ by permuting the vectors $v_i$, and on $\R^{mn}=(\R^m)^n$ by permuting the coordinates. As a $\Sigma_n$-representation, $\R^{mn}$ decomposes as $W^m \oplus \ep^m$, where $\ep^m$ includes as the diagonal in $(\R^m)^n$. We define $\tilde{F}$ as the projection onto $W^m$ and observe 
\begin{prop}
 $\tilde{F}(v_1,\cdots,v_n)=0$  iff $f(v_1)=f(v_2)=\cdots =f(v_n)$.
\end{prop}  
Using these notations, we have the following theorem. 
\begin{theorem}\label{Kakp}
1) Let $p$ be an odd prime. Then $l(m,p)\leq (\floor{\frac{m}{2}}+1)(p-1) +1$. \\ 
2) $l(m,2)\leq m+1$ if $m$ is even, and $\leq m+2$ if $m$ is odd. 
\end{theorem}

\begin{proof} 
The second statement follows from Theorem \ref{indc2} where it is shown that if $l$ is odd there is no $C_2$-map from $V_2\R^l \to S((l-1)\sigma)$. For the first statement, we use the inclusion $C_p \subset \Sigma_p$ given by the $p$-cycle $(1,\cdots, p)$, so that $W$ restricts to the reduced regular representation $\bar{\rho}$ of $C_p$. Then, $S(\bar{\rho}^m)$ is a free $C_p$-space whose underlying space is a sphere of dimension $m(p-1)-1$. It follows that $\Index_{C_p}(S(\bar{\rho}^m)) = \langle v^{\frac{m(p-1)}{2}}\rangle $. 

On the other hand if $l \geq (\floor{\frac{m}{2}}+1)(p-1) +1$, $\ceil{\frac{l}{p-1}}-1$ is at least $(\floor{\frac{m}{2}}+1)>\frac{m}{2}$. Therefore $\langle  v^{\frac{m(p-1)}{2}}\rangle$ is not contained in $\Index_{C_p}(V_p\R^l)= \langle v^{(\ceil{\frac{l}{p-1}}-1)(p-1)}\rangle$ (by Theorem \ref{indV}). The result follows.     
\end{proof}
Theorem \ref{Kakp} may be restated in the following manner for the question stated above. 
\begin{cor} \label{Kakp2} 
1) Let $f: S^{l-1} \to \R^m$  for $l\geq m+1$ if $m$ is even, and $l\geq m+2$ if $m$ is odd. Then there are orthogonal unit vectors $v,w$ in $\R^l$ such that $f(v)=f(w)$. \\ 
2) Let $f:S^{l-1} \to \R^m$ for $l \geq (\floor{\frac{m}{2}}+1)(p-1)+1$ for an odd prime $p$. Then, there are orthogonal unit vectors $v_1,\cdots,  v_p$ in $\R^l$ such that $f(v_1)=f(v_2)=\cdots = f(v_p)$. 
\end{cor}

The next result is about prime powers $n$. 
\begin{thm} \label{Kakpn}
Let $f:S^{l-1}\to \R^m$ for $l\geq (\frac{m}{2}+1)(p^k-1)  +1$ for an odd prime $p$, then there are $p^k$ orthogonal vectors $v_1,\cdots, v_{p^k}$ in $\R^l$ such that $f(v_1)=f(v_2) =\cdots = f(v_{p^k})$. That is, $l(m, p^k)\leq (\frac{m}{2}+1)(p^k-1)  +1$. 
\end{thm}

\begin{proof} 
We will rule out the existence of $\Sigma_{p^k}$-maps $V_{p^k}\R^l \to S(W^m)$. We use the inclusion $C_p^k \subset \Sigma_{p^k}$ induced by the action of $C_p^k$ on itself by left multiplication. The standard representation $W$ of $\Sigma_{p^k}$ restricts to the reduced regular representation $\bar{\rho}$ of $C_p^k$. Our method involves a comparison of Fadell-Husseini indices of $V_{p^k}(\R^l)$ and $S(W^m)$ via the commutative diagram of fibrations. 
 \[
\xymatrix{
       V_{p^k}\R^l\ar[rr] \ar[d]       & & S(W^{m})\ar[d]  \\
        (V_{p^k}\R^l)_{hC_p^k} \ar[rr] \ar[d] & & S(W^{m})_{hC_p^k} \ar[d]  \\  
        BC_p^k \ar@{=}[rr]  & & BC_p^k \\ } \] 
 Since $S(W^{m})^{C_p^k}=\phi$, the localization theorem (proposition \ref{1P}) implies that the map $H^*(BC_p^k;\Z/p)\to H^*(S(W^{m})_{hC_p^k};\Z/p)$ is not injective. Also observe that in the spectral sequence of the right vertical spectral sequence, the only non-trivial differential is $d_{m(p^k-1)}$ which is determined by $d_{m(p^k-1)}(\ep_{m(p^k-1)-1})$. The conclusion from the localization theorem implies that this is non-trivial, implying $\Index_{C_p^k}(S(W^{m})) \cap H^{m(p^k-1)}(BC_p^k;\Z/p)\neq \varnothing$.

From Theorem \ref{indV1}, we observe that the elements of $\Index_{C_p^k}(V_{p^k}\R^l )$ in the lowest degree is generated by monomials in $v_i$ of degree $p-1$. This is also the image of the first non-trivial differential in the spectral sequence associated to $V_{p^k}\R^l \to (V_{p^k}\R^l)_{hC_p^k} \to BC_p^k$. This is both $\geq 2(l-p^k+1)$ and divisible by $2(p-1)$.  This implies that the lowest degree terms are in degree at least $2(\ceil{\frac{l-p^k+1}{p-1}})(p-1)$.  
We compute 
\begin{align*} 
2\ceil{\frac{ (\frac{m}{2}+1)(p^k-1)  +1 -(p^k-1)}{p-1}}(p-1) & = 2\ceil{\frac{ \frac{m}{2}(p^k-1)  +1}{p-1}}(p-1) \\ 
                                                   & = 2(p-1) \ceil{\frac{m}{2}(1+p+\cdots + p^{k-1})+\frac{1}{p-1}} \\
                                                   &=\begin{cases} 2(p-1)(\frac{m}{2}(1+p+\cdots+p^{k-1})+1) & \mbox{if } $mk$ \mbox{ is even}     \\ 
                                                                            2(p-1)(\frac{m}{2}(1+p+\cdots+p^{k-1})+\frac{1}{2}) & \mbox{if } $mk$ \mbox{ is odd}                                                    \end{cases} \\
                                                                            & > m(p^k -1).
\end{align*}
This implies that under the given hypothesis, $\Index_{C_p^k}(V_{p^k}\R^l)$ does not contain $\Index_{C_p^k}S(W^m)$ and therefore, there cannot be any $C_p^k$-map $V_{C_p^k}\R^l \to S(W^m)$. 
\end{proof}

We next discuss the complex and quaternionic analogue of these results. For these results, we get the stronger consequence that orthogonal vectors in complex or quaternionic sense are mapped to the same value. We have the following results in this case. 

\begin{thm} \label{KakCH} 
A) Let $f:S(\C^l)=S^{2l-1} \to \R^m$ be a map. Then we have\\
1) Suppose that $2l\geq m+2$ if $m \not\equiv 2 \pmod{4}$, and $2l\geq m+4$ for $m\equiv 2 \pmod{4}$.  Then there are orthogonal unit vectors $v,w$ in $\C^l$ such that $f(v)=f(w)$.\\
2) Suppose that $l \geq (\floor{\frac{m}{2}}+1)(p-1)+1$ for an odd prime $p$. Then, there are orthogonal unit vectors $v_1,\cdots,  v_p$ in $\C^l$ such that $f(v_1)=f(v_2)=\cdots = f(v_p)$. \\
3) Suppose that $l\geq (\frac{m}{2}+1)(p^k-1)  +1$ for an odd prime $p$. Then, there are $p^k$ orthogonal vectors $v_1,\cdots, v_{p^k}$ in $\C^l$ such that $f(v_1)=f(v_2) =\cdots = f(v_{p^k})$. \\
B) Let $f:S(\Hyp^l)=S^{4l-1} \to \R^m$ be a map. Then we have\\
1) Suppose that $4l \geq m+4$ if $m\not\equiv 4 \pmod{8}$, and $4l\geq m+8$ if $m\equiv 4 \pmod{8}$. Then there are orthogonal unit vectors $v,w$ in $\Hyp^l$ such that $f(v)=f(w)$. \\
2) Suppose that 
$$l \geq \begin{cases} \frac{1}{2}(\floor{\frac{m}{2}}+1)(p-1)+1 & \mbox{if } p\mid \floor{\frac{m}{2}} +1 \\
              \frac{1}{2}(\floor{\frac{m}{2}}+2)(p-1)+1 & \mbox{if } p\nmid \floor{\frac{m}{2}} +1 \end{cases} $$ 
for an odd prime $p$. Then, there are orthogonal unit vectors $v_1,\cdots,  v_p$ in $\Hyp^l$ such that $f(v_1)=f(v_2)=\cdots = f(v_p)$. \\
3) Suppose that $l\geq \frac{1}{2}((\frac{m}{2}+2)(p^k-1)  +1)$ for an odd prime $p$. Then, there are $p^k$ orthogonal vectors $v_1,\cdots, v_{p^k}$ in $\Hyp^l$ such that $f(v_1)=f(v_2) =\cdots = f(v_{p^k})$.
\end{thm} 

\begin{proof}
We proceed as in the real case. From $f:S^{dl-1}\to \R^m$ for $d=2$ or $4$, we may construct $F: W_{l,n} \to (\R^m)^n$ in the case $d=2$ and $F: X_{l,n}\to (\R^m)^n$ in the case $d=4$ by the formula 
$$F(v_1,\cdots, v_n)= (f(v_1),\cdots,f(v_n)).$$
The map $F$ is $\Sigma_n$-equivariant, and $(\R^m)^n$ splits as a $\Sigma_n$-representation into $W^m\oplus \ep^m$ where $W$ is the standard representation of $\Sigma_n$. We define $\tilde{F}$ to be the projection of $F$ to $W^m$, and note that $\tilde{F}(v_1,\cdots, v_n)=0$ if and only if $f(v_1)=\cdots = f(v_n)$ as in Proposition \ref{stdred}. If there is no such tuple of orthogonal vectors, we may normalize the values of $\tilde{F}$ to obtain a $\Sigma_n$-equivariant map $W_{l,n} \to S(W^m)$ if $d=1$, or a $\Sigma_n$-equivariant map $X_{l,n}\to S(W^m)$. Under the given conditions on $l$ and $m$ we rule out the existence of such maps by restrictions to subgroups $C_p$ or $C_p^n$ as in the real case.

The statements A1 and B1 follow from the bound on the $C_2$-index in Theorem \ref{indc2}. It implies that there is no $C_2$ map $W_{l,2} \to S((2l-2)\sigma)$ and there is no $C_2$-map $X_{l,2} \to S((4l-4)\sigma)$ if $l$ is odd. This implies A1 in the case $4 \mid m$ and B1 in the case $8 \mid m$. The remaining statements in A1 follow from the fact that $W_{l,2}$ is $(2l-4)$-connected and $S(m\sigma)$ is $(m-2)$-connected, so that Proposition \ref{conn} implies the conclusions in the other cases (note that if $m$ is odd, $2l\geq m+2$ already implies $2l\geq m+3$). The remaining statements of B1 also follow from Propostion \ref{conn} using the fact that $X_{l,2}$ is $(4l-6)$-connected. 

The statement A2 follows in exactly the same manner as Theorem \ref{Kakp} in the $p$ odd case, as the index of the complex Stiefel manifold has the same formula as the real Stiefel manifolds at odd primes (Compare Theorems \ref{indV} and \ref{indW}). We may deduce B2 from the index computations in Theorem \ref{indX}  by observing that in the case $p\mid \floor{\frac{m}{2}}+1$, 
$$\ceil{\frac{2l}{p-1}}-1 \geq \ceil{\frac{(\floor{\frac{m}{2}}+1)(p-1)+2}{p-1}}-1=\floor{\frac{m}{2}}+1 > \frac{m}{2},$$
and  in the case $p\nmid \floor{\frac{m}{2}}+1$, 
$$\ceil{\frac{2l}{p-1}}-2 \geq \ceil{\frac{(\floor{\frac{m}{2}}+2)(p-1)+2}{p-1}}-2=\floor{\frac{m}{2}}+1 > \frac{m}{2},$$
and then complete the proof using the  index computations for $S(\bar{\rho}^m)$ in Theorem \ref{Kakp}.

With respect to the statement A3, we compare the index computation for the complex Stiefel manifold in Theorem \ref{indWXn} with the computation for the real Stiefel manifold in Theorem \ref{indV1}. We proceed as in Theorem \ref{Kakpn}, to observe that the first non-trivial differential in the Serre spectral sequence associated to the fibration 
$$W_{l,p^k} \to (W_{l,p^k})_{hC_p^k} \to BC_p^k,$$  
has image in degree $\geq 2(l-p^k+1)$ and divisible by $2(p-1)$. In this case the result follows in an analogous manner as in Theorem \ref{Kakpn}. 

Observe now that for the fibration 
$$X_{l,p^k} \to (X_{l,p^k})_{hC_p^k} \to BC_p^k,$$
Theorem \ref{indWXn} implies that the image of the first non-trivial differential is in degree which is both $\geq 4(l-p^k+1)$ and divisible by $2(p-1)$.  This implies that the lowest degree terms are in degree at least $2(\ceil{\frac{2(l-p^k+1)}{p-1}})(p-1)$. The computation at the end of the proof of Theorem \ref{Kakpn} implies that under the given condition this is $>\frac{m}{2}(p^k-1)$. This implies that $\Index_{C_p^k}S(W^m) \not\subset \Index_{C_p^k}X_{l,p^k}$ and thus, completes the proof of B3. 
\end{proof}

\end{document}